\documentclass[12pt]{amsart}
\usepackage{a4wide,graphicx}
\usepackage{color}
\usepackage{caption}
\usepackage{comment}
\usepackage{enumerate}
\usepackage[left=1in, right=1in]{geometry}
\usepackage{bbm,yfonts}
\usepackage{hyperref} 
\usepackage{wasysym}
\usepackage{subfig}
\usepackage{pifont}

\let\pa\partial  
  
\let\eps\varepsilon  

\newcommand{\N}{{\mathbb N}} 
\newcommand{\R}{{\mathbb R}}

\newcommand{\Cl}{\operatorname{Cl}}
\newcommand{\dx}{\mathrm{d}x}
\newcommand{\D}{\operatorname{D}}
\newcommand{\T}{{{\mathbb T}}}

\newcommand{\A}{{\mathbb A}}

\newcommand{\dd}{{\mathrm{d}}}

\newcommand{\entropy}{\mathcal{H}}

\newcommand{\F}{{\mathcal F}}
\newcommand{\E}{{\mathcal E}}
\newcommand{\J}{{\mathbb J}}
\newcommand{\M}{{\mathcal M}}
\renewcommand{\H}{{\mathsf H}}
\newcommand{\Fd}{{\mathcal F}_{\mathrm{d}}}
\newcommand{\Ad}{{\mathbb A}_{\mathrm{d}}}
\newcommand{\Ham}{{\mathcal H}}

\theoremstyle{plain}
\newtheorem{theorem}{Theorem}[section]   
   
\newtheorem{proposition}[theorem]{Proposition}

\theoremstyle{definition}

\theoremstyle{remark}
\newtheorem{remark}{Remark}[section]

\begin{document}

\title[~]{Well-posedness and convergence of a numerical scheme for the corrected 
Derrida-Lebowitz-Speer-Spohn equation using the Hellinger distance}

 
 \author{Mario Bukal}
 \address{University of Zagreb,
 Faculty of Electrical Engineering and Computing\newline
 Unska 3, 10000 Zagreb, Croatia}
 \email{mario.bukal@fer.hr}
 \thanks{This work has been supported by the Croatian Science
 Foundation under Grant agreement No.~7249 (MANDphy) and in part by the
bilaterial project No.~HR 04/2018 between OeAD and MZO}


 \begin{abstract}
In this paper we construct a unique global in time weak nonnegative solution to the corrected
Derrida-Lebowitz-Speer-Spohn equation, which statistically describes the interface
fluctuations between two phases in a certain spin system. The construction of the weak solution 
is based on the dissipation of a Lyapunov functional which equals to the square of the Hellinger 
distance between the solution and the constant steady state. 
Furthermore, it is shown that the weak solution 
converges at an exponential rate to the constant steady state in the Hellinger distance and thus
also in the $L^1$-norm.
Numerical scheme which preserves the variational structure of the equation is devised and 
its convergence in terms of a discrete Hellinger distance is demonstrated.
 \end{abstract}
\date{\today}

\maketitle

\section{Introduction} 
Nonlinear evolution equations with higher-order
spatial derivatives appear as approximate
models in various contexts of mathematical physics. Besides the Cahn-Hillard equation \cite{CaHi58},
the most prominent models are
various thin-film equations describing dynamics of the thickness of a thin viscous fluid film
\cite{CDGKSZ93, ODB97, Mye98, Ber98}. 
In the case of free boundary film, the dynamics is driven by the competition between the 
surface tension and another potential force like gravity, capillarity, heating, Van der Waals, etc., 
which leads to a fourth-order evolution equation. Similarly, in the case when the fluid is covered by 
a thin elastic plate, then the pressure in the fluid is balanced by the sum of the bending of the plate 
and a potential force, which eventually leads to a sixth-order evolution equation \cite{HoMa04,LPN13}.
Many other higher-order models related to modelling of isolation oxidation of silicon in
classical semiconductors \cite{King89}, approximation of quantum effects in quantum semiconductors \cite{DMR05}, 
description of Bose-Einstein condensate \cite{JPR06}, image analysis\cite{BHS09}, etc.~can be found in the literature.

In this paper we study particular
fourth-order evolution equation  
\begin{equation}\label{1.eq:eDLSS}
\pa_t u 
= -\frac{1}{2}\left(u(\log u)_{xx}\right)_{xx} + 2\delta\left(u^{3/4}(u^{1/4})_{xx}\right)_x\,,
\end{equation}
which has been derived in \cite{BGT15} as a corrected version of the well known
Derrida-Lebowitz-Speer-Spohn (DLSS for short) equation. The latter first appeared in 
\cite{DLSS91A} in the form of (\ref{1.eq:eDLSS}) with $\delta = 0$. Unknown $u$ in 
(\ref{1.eq:eDLSS}) denots 
the density function of a probability distribution which asymptotically describes 
the statistics of
interface fluctuations between two phases of spins in the anchored Toom model. 
For the later reference we call 
 (\ref{1.eq:eDLSS}) \emph{the corrected DLSS equation}. To complete the problem for 
 equation (\ref{1.eq:eDLSS})
 we assume periodic boundary conditions, i.e.~$x\in\T = [0,1)$, where
 endpoints of the interval $0$ and $1$ are identified, and we
 prescribe a nonnegative initial datum $u(0) = u_0\geq0$ a.e.~on $\T$. 
 
Since the seminal paper by Bernis and Friedmann \cite{BeFr90}, analysis of 
higher-order nonlinear evolution equations, especially thin-film equations,
has become an attractive field of interest in mathematics community. 
Existence of solutions and their qulitative properties like positivity, compact support, 
blow up, the long time asymptotics are among most important questions. 
Even the thin-film equation alone has very rich mathematical structure, which can be retrieved 
from \cite{BeGr, DGG98,
GiaOtt01, CaTo02}, and references therein. Adding lower (second) order unstable terms results 
in more complex dynamics \cite{WBB04, NoSi10}.

The original DLSS equation has been first analyzed by Bleher et.~al.~in \cite{BLS94}.
Employing the semigroup approach they proved the local in time existence of positive 
classical solutions. Moreover, they proved equivalence between strict positivity and
smoothness of the solution. The very same conclusions apply for the equation at hand.
Namely, the third-order term from (\ref{1.eq:eDLSS}) enters into the ``perturbation term''
in \cite{BLS94}, and all results apply analogously. The first construction of global
in time weak nonnegative solutions to the DLSS equation, accompanied with the long time
behavior analysis has been performed in \cite{JuPi00}. Later on many results related to the
DLSS equation have been achieved, we emphasize on \cite{JuMa08}, which generalizes the
result from \cite{JuPi00} to the multidimensional case, and \cite{GST09}, where the
gradient flow structure has been rigorously justified. Namely, it has been shown that 
the DLSS equation constitutes the gradient flow of the Fisher information functional 
with respect to the $L^2$-Wasserstein metric. Let us point out at this place that 
the third-order term in equation (\ref{1.eq:eDLSS}) also possesses a geometric structure,
it can be formally seen as a Hamiltonian flow of the Fisher information. A detailed
discussion on this is postponed to Section \ref{sec:dp}. To conclude on the well-posedness for
the DLSS equation, in \cite{Fis13} Fischer proved the uniqueness of weak solutions 
constructed by J\"ungel and Matthes in \cite{JuMa08}.


Our approach to the construction of global in time weak nonnegative solutions to 
equation (\ref{1.eq:eDLSS}) closely follows the ideas developed in \cite{JuPi00} and \cite{JuMa08}.
Unlike there, where the key source of a priori estimates is the dissipation of the 
Boltzmann entropy, our construction is based on the dissipation of the functional
\begin{equation*}
\E(u) = 2\int_\T (u - 2\sqrt{u} + 1)\dd x\,.
\end{equation*}
$\E$ is the only nontrivial ``zero-order'' functional, called {\em entropy} in further, 
for which we can formally prove the
dissipation along solutions to (\ref{1.eq:eDLSS}), see Section \ref{sec:dp} for more details.
Observe that $\E(u)$ is in fact proportional to the square of the 
Hellinger distance $\H(u,u_\infty)$ between $u$ and $u_\infty = 1$
which is the steady state of equation (\ref{1.eq:eDLSS}). More precisely, 
$ \frac14\E(u) = \frac12\int_\T (\sqrt{u} - 1)^2\dd x \equiv \H^2(u,u_\infty)$.
The key source of our a priori estimates is the following entropy production inequality
\begin{equation*}
-\frac{\dd}{\dd t}\E(u(t)) \geq 4\int_\T (u^{1/4})_{xx}^2\dd x\,\,,\quad t>0\,,
\end{equation*}
valid along smooth positive solutions of (\ref{1.eq:eDLSS}). The above inequality then 
motivates to rewrite the original equation (\ref{1.eq:eDLSS}) in a novel form in terms of
$u^{1/4}$ (cf.~\cite{JuVi07})
\begin{equation}\label{1.eq:novel}
\pa_tu = 
- 2\left(u^{1/2}\left(u^{1/4}(u^{1/4})_{xx} - (u^{1/4})_x^2\right)\right)_{xx} + 
2\delta\left(u^{3/4}(u^{1/4})_{xx}\right)_x\,.
\end{equation}
Equation (\ref{1.eq:novel}) is equivalent to (\ref{1.eq:eDLSS}) for smooth and strictly 
positive solutions.
Now we can state the first result, which we prove in Section \ref{sec:weak}.
\begin{theorem}\label{tm:exist}
Let $u_0\in L^1(\T)$ be given nonnegative function of unit mass and of finite
entropy $\E(u_0)<\infty$. Let $T>0$ be given arbitrary time horizon. 
Then there exists a unique nonnegative unit mass function $u\in W^{1,1}(0,T;H^{-2}(\T))$ 
satisfying $u^{1/4}\in L^{2}(0,T;H^{2}(\T))$ and
\begin{equation*}
\int_0^T\langle\pa_tu,\phi\rangle_{H^{-2},H^2}\,\dd t + 
2\int_0^T\!\!\!\int_\T \left(u^{1/2}\left(u^{1/4}(u^{1/4})_{xx} - 
(u^{1/4})_x^2\right)\phi_{xx} + \delta u^{3/4}(u^{1/4})_{xx}\phi_x\right)\dd x\dd t = 0
\end{equation*}
for all test functions $\phi\in L^{\infty}(0,T;H^{2}(\T))$. 
\end{theorem}

Concerning the question of the long time behaviour of weak solutions the approach
is somewhat different than usual. This is due to the lack of a ``global'' Beckner type inequality for 
the parameter required by the entropy functional $\E$, i.e.~a lack of a ``global'' entropy-entropy production
inequality. As a consequence, 
we cannot obtain time decay of the entropy functional $\E$ at a universal exponential rate.
Therefore, we employ appropriate ``asymptotic'' Beckner type inequality proved in \cite{CDGJ06},
which will eventually provide an exponential time decay of the entropy functional $\E$, but at a rate
depending on the chosen initial datum $u_0$, or more precisely on $\E(u_0)$.

Relation between the $L^1$ distance and the Hellinger distance, $\|u - v\|_{L^1(\T)} \leq 2\H(u,v)$,
then implies the following result.  
\begin{theorem}
\label{tm:ltb}
Let $u_0\in H^1(\T)$ be a nonnegative function of unit mass such that $\E(u_0) < +\infty$.
Then the weak solution constructed in Theorem \ref{tm:exist} converge at an
exponential rate to the constant steady state $u_\infty = 1$ in the norm
\begin{equation*}
\|u(t) - u_\infty\|_{L^1(\T)} \leq \sqrt{\E(u_0)}e^{-\lambda t}\,,\quad t> 0\,,
\end{equation*}
where $\lambda = 4\pi^4/(1+C\sqrt{\E(u_0)}$ and $C>0$ depends only on $\E(u_0)$.
\end{theorem}
 
It is very important that numerical schemes preserve some important features of 
equations of mathematical physics, for example positivity, conservation of mass, 
dissipation of certain functionals etc. Such schemes are then expected to be more
reliable and robust to capture true behaviour of solutions, especially in 
the long run simulations. There are many such schemes in the literature devised
for the original DLSS equation \cite{JuPi01, CJT03, DMM10,BEJ14, MaMa16,MaOs17}. 
Here we discuss a discrete variational derivative (DVD)
scheme \cite{FuMa10}, which is a slight modification of the scheme proposed in \cite{BEJ14} for the original DLSS
equation. DVD schemes are finite difference type schemes which respect the variational
structure of equations. For smooth positive solutions, equation
(\ref{1.eq:eDLSS}) can be rewritten in an equivalent (variational) form 
\begin{equation}\label{1.eq:gf_form}
\pa_t u = -\left(u\left(\frac{(\sqrt{u})_{xx}}{\sqrt{u}}\right)_x\right)_x 
+ \delta\sqrt{u}(\sqrt{u})_{xxx}\,,
\end{equation} 
which can be further written as 
\begin{equation}\label{1.eq:gf_form2}
\pa_t u = \left(u\left(\F'(u)\right)_x\right)_x - \delta\sqrt{u}\left(\sqrt{u}\F'(u)\right)_x\,,
\end{equation}
where $\F'(u) = -(\sqrt{u})_{xx}/\sqrt{u}$ denotes the variational derivative of the
Fisher information. Form (\ref{1.eq:gf_form2}) of the equation obviously gives the dissipation
of the Fisher information, and this is precisely the form of the $L^2$-Wasserstein gradient flow
being justified in \cite{GST09}. The main idea of DVD schemes is to construct a discrete
analogue of (\ref{1.eq:gf_form2}), which will ensure the dissipation of the discrete version
of the Fisher information on the discrete level.

However, we will not approximate directly (\ref{1.eq:gf_form2}), instead we approximate
\begin{equation}\label{1.eq:gf_form3}
\pa_t \sqrt{u} = \frac{1}{2\sqrt{u}}\left(u\left(\F'(u)\right)_x\right)_x 
- \frac{\delta}{2}\left(\sqrt{u}\F'(u)\right)_x\,.
\end{equation}
Advantage of using this form has been already addressed in \cite{BLS94} and \cite{Fis13} for $\delta=0$.
The main cause lies in the monotonicity of the operator
\begin{equation*}
\A(v) = \frac{1}{v}\left(v^2\left(\frac{v_{xx}}{v} \right)_x\right)_x\,,
\end{equation*}
which in our case will be the key ingredient for establishing the error estimates for the numerical scheme. 

Let $\T_N = \{x_i\ : \ i=0,\ldots,N,\ x_0\cong x_N\}$ denotes an
equidistant grid of mesh size $h$ on the one dimensional torus $\T
\cong [0,1)$ and let the vector $U^k\in \R^N$ with components $U_i^k$, $i=0,\ldots,N-1$,
$k\geq0$, approximates solution $u(t_k,x_i)$ at point $x_i\in\T_N$ and time $t_k = k\tau$, where
$\tau>0$ denotes the time step. 
Given $U^0\in\R^N_+$ the DVD scheme for equation (\ref{1.eq:gf_form3}) is defined 
by the following nonlinear system with unknowns
$V^{k+1}_i = \sqrt{U^{k+1}_i}$:
\begin{align}\label{1.sh.dvdm}
\frac{1}{\tau}(V_i^{k+1} - V_i^k) &= \frac{1}{2W^{k+1/2}_i}\delta_i^+\left(W^{k+1/2}_iW^{k+1/2}_{i-1}\delta_i^-
\left(\delta \Fd(W^{k+1/2})_i\right)\right) \\\nonumber
&\qquad - 
\frac{\delta}{2} \delta_i^{\langle1\rangle}\left(W^{k+1/2}_i\delta \Fd(W^{k+1/2})_i\right)\,,
\end{align}
for all $i=0,\ldots,N-1\,,\ k\geq 0\,,$ where $W^{k+1/2} = (V^{k+1} + V^k)/2$ and 
$\delta \Fd(W)_i = -\delta_i^{\langle2\rangle}W/W_i$ denotes the discrete variational derivative 
of the discrete Fisher information $\Fd$ defined by (\ref{def:discreteFisher}). Above $\delta_i^\pm$,
$\delta_i^{\langle1\rangle}$ and $\delta_i^{\langle2\rangle}$ denote finite difference operators 
precisely introduced in section \ref{sec:41}. 
Note that DVD scheme (\ref{1.sh.dvdm}) imitates equation (\ref{1.eq:gf_form3}) on the discrete level, 
and particular combination of discrete operators is justified by the following result.
\begin{theorem} \label{tm:dvds}
Let $N\in \N$, $\tau>0$ and $U^0\in \R^N_+$ be some
nonnegative initial datum satisfying $h\sum_{i=0}^{N-1}U_i^0 = 1$. The
scheme (\ref{1.sh.dvdm}) is consistent of order $O(\tau) + O(h^2)$ with respect to the 
time-space discretization, solutions $U^k$, $k\geq1$, are nonnegative by construction, 
satisfy $h\sum_{i=0}^{N-1}U_i^k = 1$, and the discrete
Fisher information is nonincreasing, i.e.~$\Fd(U^{k+1})\leq \Fd(U^k)$ for all $k\geq0$. 
Furthermore, there exists a constant $C>0$, independent of $\tau$ and $h$, such that
\begin{equation}\label{ineq:error_estimate}
h\sum_{i=0}^{N-1}\left(\sqrt{u^k_i} - \sqrt{U^k_i}\right)^2 \leq C\frac{\tau^2 + h^4}{1-\tau}
\,\quad\text{for all } k\geq 1\,,
\end{equation} 
where $u^k\in\R^N$ represents values of sufficiently smooth solutions to (\ref{1.eq:gf_form3}) 
at grid points $\T_N$ at time $t_k$.

\end{theorem}
\begin{remark}
Inequality (\ref{ineq:error_estimate}) provides a quantitative error estimate and thus convergence
of the DVD scheme. Defining a discrete analogue of the Hellinger distance as
\begin{equation}\label{def.discreteHell}
\H_{\dd}(U,V)^2 = \frac{h}{2}\sum_{i=0}^{N-1}\left(\sqrt{U_i} 
- \sqrt{V_i}\right)^2\,,\quad\text{for }U,V\in\R^N_+\,,
\end{equation}
inequality (\ref{ineq:error_estimate}) can be interpreted as 
$\sup_{k\in\N}\H_{\dd}(u^k,U^k)\leq C(\tau+h^2)$ for some $C>0$.
\end{remark}

The paper is organized as follows. In Section \ref{sec:dp} we discuss some formal
dissipation pro\-per\-ties and the geometric structure of the equation. Section \ref{sec:weak}
is devoted to proofs of Theorems \ref{tm:exist} and \ref{tm:ltb}, while in the last section we 
introduce the numerical scheme and prove its pro\-per\-ties summarized in Theorem \ref{tm:dvds}.

\section{Formal dissipation properties and geometric structure}\label{sec:dp}

\subsection{Entropy production estimates}
Before we undertake a thorough analysis on the wellposedness, let us discuss 
some formal dissipation properties of equation (\ref{1.eq:eDLSS}), 
which will be in the heart of rigorous proofs.
For this purpose we assume the existence of smooth and strictly positive solutions to equation 
(\ref{1.eq:eDLSS}) 
and consider a parametrized family of functionals of the form
\begin{align}\label{2.def:entropy}
\E_\alpha(u) &= \frac{1}{\alpha(\alpha-1)}\int_{\T}(u^{\alpha} - \alpha u + \alpha - 1)\dd x\,,\quad \alpha\neq 0,1\,,\nonumber\\
\E_1(u) &= \int_{\T}(u\log u - u + 1)\dd x\,,\quad \alpha = 1\,,\\
\E_0(u) &= \int_\T (u - \log u)\dd x\,,\quad \alpha = 0\,.\nonumber
\end{align}
In particular, we are looking for those functionals satisfying the so called Lyapunov property, 
i.e.~$(\dd/\dd t)\,\E_\alpha(u(t)) \leq 0\,$ along solutions to 
(\ref{1.eq:eDLSS}) for all $t>0$. Although having the 
opposite sign, functionals (\ref{2.def:entropy}) are often named {\em entropies}
due to their connection to the Boltzmann-Shannon entropy $\entropy(u) = -\E_1(u)$ and 
Tsallis entropies $\mathcal T_{\alpha}(u) = -\alpha\E_\alpha(u)$. 
For smooth and positive solutions we can write equation (\ref{1.eq:eDLSS}) 
in an equivalent polynomial representation
\begin{equation*}
\pa_tu = \left(u P_\delta\left(\frac{u_x}{u},\frac{u_{xx}}{u},\frac{u_{xxx}}{u}\right)\right)_x\,,
\end{equation*}
where the polynomial $P_\delta$ is given by 
\begin{equation}
P_\delta(\xi_1,\xi_2,\xi_3) = -\frac12\xi_3 + \xi_1\xi_2-\frac12\xi_1^3 
+ \delta\left(\frac12\xi_2 - \frac38\xi_1^2\right)\,.
\end{equation}

Calculating the entropy production we find
\begin{align}\label{2.eq:ent_prod}
-\frac{\dd}{\dd t}\E_\alpha(u) = \int_{\T} u^{\alpha}\left(\frac{u_x}{u}\right)
P_\delta\left(\frac{u_x}{u},\frac{u_{xx}}{u},\frac{u_{xxx}}{u}\right)\dd x
=: \int_\T u^{\alpha}S_\delta\left(\frac{u_x}{u},\frac{u_{xx}}{u},\frac{u_{xxx}}{u}\right)\dd x\,,
\end{align}
with the polynomial $S_\delta$ given by
$$
S_\delta(\xi) = \xi_1P_\delta(\xi) = -\frac12\xi_1\xi_3 + \xi_1^2\xi_2-\frac12\xi_1^4 
+ \delta\left(\frac12\xi_1\xi_2 - \frac38\xi_1^3\right)\,.
$$
We are now looking for all $\alpha\in\R$ such that the
integral inequality $-(\dd/\dd t) \E_\alpha(u(t)) \geq 0$ holds.
In order to assert the integral inequality, we systematically use integration by parts formulae 
and transform integrands using their polynomial representation. 
Observe that that equation (\ref{1.eq:eDLSS}) itself, and thus polyinomial $S_\delta$ as well,
do not possess a homogeneity properties like those in \cite{JuMa06}. However, the method of
algorithimic construction of entropies proposed in \cite{JuMa06} can be adjusted to the equation
at hand. First we identify elementary integration by parts formulae, which
are represented by so called {\em shift polynomials} $T_i(\xi)$:
\begin{align*}
\int_\T\left(u^\alpha\left(\frac{u_x}{u}\right)^3\right)_x\dd x &= 
\int_\T u^\alpha T_1\left(\frac{u_x}{u}, \frac{u_{xx}}{u}, \frac{u_{xxx}}{u}, \frac{\pa_x^4u}{u}\right)\dd x\,, 
\quad T_1(\xi) = 3\xi_1^2\xi_2 + (\alpha-3)\xi_1^4\,,\\ 
\int_\T\left(u^\alpha\frac{u_x}{u}\frac{u_{xx}}{u}\right)_x\dd x &= 
\int_\T u^\alpha T_2\left(\frac{u_x}{u}, \frac{u_{xx}}{u}, \frac{u_{xxx}}{u}, \frac{\pa_x^4u}{u}\right)\dd x\,, 
\quad T_2(\xi) = \xi_2^2 + (\alpha-2)\xi_1^2\xi_2 + \xi_1\xi_3\,,\\
\int_\T\left(u^\alpha\frac{u_{xxx}}{u}\right)_x\dd x &= 
\int_\T u^\alpha T_3\left(\frac{u_x}{u}, \frac{u_{xx}}{u}, \frac{u_{xxx}}{u}, \frac{\pa_x^4u}{u}\right)\dd x\,, 
\quad T_3(\xi) = (\alpha-1)\xi_1\xi_3 + \xi_4\,,\\ 
\int_\T\left(u^\alpha \left(\frac{u_x}{u}\right)^2\right)_x\dd x &= 
\int_\T u^\alpha T_4\left(\frac{u_x}{u}, \frac{u_{xx}}{u}, \frac{u_{xxx}}{u}, \frac{\pa_x^4u}{u}\right)\dd x\,, 
\quad T_4(\xi) = (\alpha-2)\xi_1^3 + 2\xi_1\xi_2\,,\\
\int_\T\left(u^\alpha\frac{u_{xx}}{u}\right)_x\dd x &= 
\int_\T u^\alpha T_5\left(\frac{u_x}{u}, \frac{u_{xx}}{u}, \frac{u_{xxx}}{u}, \frac{\pa_x^4u}{u}\right)\dd x\,, 
\quad T_5(\xi) = (\alpha-1)\xi_1\xi_2 + \xi_3\,.
\end{align*}
All other integration by parts formulae can be obtained as linear combinations of these.

Observe that 
\begin{equation*}
\int_\T u^\alpha T_i\left(\frac{u_x}{u}, \frac{u_{xx}}{u}, 
\frac{u_{xxx}}{u}, \frac{\pa_x^4u}{u}\right)\dd x = 0\,,\qquad i=1,\ldots,5\,,
\end{equation*}
thus, adding an arbitrary linear combination of 
the above integrals to (\ref{2.eq:ent_prod}) does not change the value of the
entropy production, but only changes the integrand, i.e.~for any $c_1,\ldots, c_5\in\R$
\begin{align}\label{2.eq:ent_prod_transformed}
-\frac{\dd}{\dd t}\E_\alpha(u) =
 \int_\T u^{\alpha}\left(S_\delta  + \sum_{i=1}^5 c_iT_i\right)
 \left(\frac{u_x}{u},\frac{u_{xx}}{u},\frac{u_{xxx}}{u}, \frac{\pa_x^4u}{u}\right)\dd x\,.
\end{align}
Clearly, sufficient condition for nonnegativity of an integral is pointwise 
nonnegativity of the integrand function. This turns the integral inequality 
$-(\dd/\dd t) \E_\alpha(u(t)) \geq 0$ into the polynomial decision problem:
\begin{align*}
(\exists c_1,\ldots, c_5\in\R)\,,(\forall \xi\in\R^4)\,,\quad 
\left(S_\delta + \sum_{i=1}^5c_iT_i\right)(\xi)\geq0\,,
\end{align*}
which is according to the Tarski \cite{Tar51} always decidable (solvable).
After straightforward calculations the decision problem amounts to
\begin{align}
(\exists c_1, c_4\in\R)\,,(\forall \xi\in\R^4)\,,\quad 
\frac12\xi_2^2 & + \left(3c_1 + \frac{\alpha}{2}\right)\xi_1^2\xi_2 
+ \left((\alpha-3)c_1 - \frac12\right)\xi_1^4\\
&\quad + \left(2c_4 + \frac{\delta}{2}\right)\xi_1\xi_2 \nonumber
+ \left((\alpha-2)c_4 - \frac{3\delta}{8}\right)\xi_1^3\geq0\,,
\end{align}
whose solution can be resolved, for instance with \verb1Wolfram Mathematica1,
to
\begin{equation}\label{2.eq:sol_dp}
\delta = 0\ \text{\sc and}\ 0\leq \alpha \leq \frac32 \quad \text{\sc or}
\quad \delta > 0\ \text{\sc and}\ \alpha = \frac12\,.
\end{equation}
First part of the sentence (\ref{2.eq:sol_dp}), $\delta = 0\ \text{\sc and}\ 0\leq \alpha \leq 3/2$,
is the well known result for the original DLSS equation \cite{JuMa06}, while the second part,
$\delta > 0\ \text{\sc and}\ \alpha = 1/2$,
concerns our equation (\ref{1.eq:eDLSS}), and provides only 
\begin{equation*}
\E_{1/2}(u) = 2\int_\T (u - 2\sqrt{u} + 1)\dd x
\end{equation*} 
as an entropy (Lyapunov functional) for equation (\ref{1.eq:eDLSS}). 

By means of the same method \cite{JuMa06} as briefly presented above, 
the entropy production can be further estimated
from bellow by a positive nondegenerate functional as follows:
\begin{equation*}
-\frac{\dd}{\dd t}\E_{1/2}(u(t)) \geq 4\int_\T (u^{1/4})_{xx}^2\dd x\,\,,\quad t>0\,
\end{equation*}
along smooth positive solutions to equation (\ref{1.eq:eDLSS}). In this way we also 
proved the following key estimate for the construction of weak solutions (cf.~\cite{JuMa08}).
\begin{proposition}[Entropy production estimate] Let $u\in H^2(\T)$ be strictly positive,
then the following functional inequality holds 
\begin{equation}\label{2.ineq:epi}
-\int_\T u(u^{-1/2})_{xx}(\log u)_{xx}\geq 4\int_\T (u^{1/4})_{xx}^2\dd x\,.
\end{equation}
\end{proposition}

Although the above dissipation results for equation (\ref{1.eq:eDLSS}) seem to be a poor
in comparison with the dissipation structure of the original DLSS equation,
it turns out that estimate (\ref{2.ineq:epi}) is sufficient for 
the  construction of global weak sloutions, which
we perform in a subsequent section.

\subsection{Geometric structure}
Besides the entropy $\E_{1/2}$, there is another distinguished Lyapunov 
functional for the dynamics of (\ref{1.eq:eDLSS}), the {\em Fisher information}, which
is defined by
\begin{equation}\label{2.def:Fi}
\F(u) =  \int_{\T} \left(\sqrt{u}\right)_x^2 \dd x\,.
\end{equation}
The Lyapunov property of the Fisher information is directly seen from the 
following equivalent (for smooth positive solutions) formulation of equation (\ref{1.eq:eDLSS})
\begin{equation}\label{2.eq:gf_form}
\pa_t u = -\left(u\left(\frac{(\sqrt{u})_{xx}}{\sqrt{u}}\right)_x\right)_x 
+ \delta\sqrt{u}(\sqrt{u})_{xxx}\,,
\end{equation} 
which can be further written as 
\begin{equation}\label{2.eq:gf_form2}
\pa_t u = \left(u\left(\F'(u)\right)_x\right)_x - \delta\sqrt{u}\left(\sqrt{u}\F'(u)\right)_x\,,
\end{equation}
where $\F'(u) = -(\sqrt{u})_{xx}/\sqrt{u}$ denotes the variational derivative of the
Fisher information.

The first term on the right hand side in (\ref{2.eq:gf_form2}) has the well known structure
of the gradient flow with respect to the $L^2$-Wasserstein metric. This structure has been
rigorously justified and exploited for the original DLSS equation posed on the whole space \cite{GST09}.
It has been shown that the DLSS equation constitutes the gradient flow of the 
Fisher information with respect to the $L^2$-Wassersten metric.

We find it a remarkable fact that the second term on the right-hand side in (\ref{2.eq:gf_form2}) 
formally possesses the structure of a 
{\em Hamiltonian flow} of the Fisher information, which we discuss more in detail bellow.
Hence, equation (\ref{2.eq:gf_form}) can be formally written as a mixture flow, i.e.~the sum of 
the gradient and the Hamiltonian flow
\begin{equation}\label{eq:gfH}
\pa_t u = -\nabla_{W_2}\F(u) + X_\F(u)\,.
\end{equation}

\subsubsection{Symplectic structure of the third-order term} Let us briefly discuss 
the structure of the third-order term in (\ref{2.eq:gf_form}), i.e.~we only consider equation
\begin{equation}\label{eq:dispersion}
\pa_t u = \sqrt{u}(\sqrt{u})_{xxx}\,.
\end{equation}

First of all, direct formal calculations reveal that all functionals
\begin{equation*}
\F_n(u) =  \int_\T \left(\pa_x^n\sqrt{u}\right)^2 \dd x\,,\quad n\in\N_0\,,
\end{equation*}
are constants of motion (first integrals) for equation (\ref{eq:dispersion}).
Namely,
\begin{align*}
\frac{\dd}{\dd t}\F_n(u(t)) 
& = (-1)^{n\bmod 2}\int_\T \pa_x^{2n}\sqrt{u}\pa_x^3\sqrt u\,\dd x 
 = -\frac{1}{2}\int_\T \pa_x\left(\pa_x^{n+1}\sqrt u\right)^2\dd x = 0\,.
\end{align*}
Note that for the full equation (\ref{2.eq:gf_form}), 
only $\F_0$ (mass) is conserved and $\F_1$ (Fisher information) is dissipated, while
for all other $n\geq2$ the Lyapunov property of functionals $\F_n$ is an open question.
Using the method of systematic integration 
by parts like above, production of the Fisher information can also be bounded from bellow
as follows
\begin{align}
-\frac{\dd}{\dd t}\F(u(t)) 
\geq \kappa\int_\T \left((\sqrt{u})_{xxx}^2 + (\sqrt[6]{u})_x^6\right)\dd x\,, 
\label{ineq:epF1}
\end{align}
for some $\kappa > 0$, which can be explicitely determined.

Following \cite{AGS08} let $\M$ denotes the set of smooth positive densities on $\T$ 
(Radon derivatives w.r.t.~the Lebesque measure).
For every $u\in \M$, let
\begin{equation}\label{def:tang_space}
T_u\M := \Cl_{L^2(u\dd x)}\{\pa_x\phi\ |\ \phi\in C^\infty(\T)\}\,
\end{equation}
denotes the tanget space at $u$, and by $T\M$ we denote the tangent bundle. There is a natural
orthogonal decomposition of $L^2(\T,u\dx)$ according to
\begin{equation*}
L^2(\T,u\dd x) = T_u\M \oplus [T_u\M]^\perp\,,
\end{equation*}
where $[T_u\M]^\perp = \{v\in L^2(\T,u\dx)\ :\ \pa_x(uv) = 0\}$, and let 
$\pi_u : L^2(\T, u\dd x) \to T_u\M$
denotes the orthogonal projection.
For fixed $u\in \M$ we define operator $\J_u : C^\infty(\T) \to (C^\infty(\T))^*$
by
\begin{equation}\label{def:J}
\langle\J_u\phi,\psi\rangle := -\int_\T \sqrt{u}\pa_x(\sqrt{u}\phi)\psi\,\dd x\,,\quad
\forall\psi\in C^\infty(\T)\,.
\end{equation}
If $u$ is positive and smooth enough, $\J_u\phi$ is given by its $L^2$-representative 
\begin{equation*}
\J_u\phi = -\sqrt u\pa_x(\sqrt{u}\phi)\in L^2(\T,u\dd x)\,.
\end{equation*}
In such a case we define the subbundle $\hat{T}\M$ according to
\begin{equation*}
\hat T_u\M := \{\pi_u(\J_u\phi)\ : \ \phi\in C^\infty(\T)\}\,,\quad u\in\M\,,
\end{equation*}
and on that bundle we define differential 2-form $\omega_u : \hat T_u\M\times \hat T_u\M\to \R$, by
\begin{equation}\label{def:diform}
\omega_u(\xi_1,\xi_2)	
:= \langle\J_u\phi_1,\phi_2\rangle = -\int_\T \sqrt{u}\phi_2\pa_x(\sqrt{u}\phi_1)\dd x\,,
\end{equation}
where $\xi_i = \pi_u(\J_u\phi_i)$ for $i=1,2$. 

Observe that for every $u\in\M$, 2-form $\omega_u$ is bilinear and skew-symmetric. 
Also, for every $u\in\M$ and $0\neq\xi = \pi_u(\J_u\phi)\in \hat T_u\M$, choosing 
$\eta_\xi = \pi_u(\phi)\neq0$, it follows
\begin{equation*}
\omega_u(\eta_\xi,\xi) 
= \|\phi\|_{L^2}^2 \neq 0\,,
\end{equation*}
which shows that $\omega_u$ is nondegenerate.
In order to prove that $\omega_u$ is symplectic, it remains to check that it is exact,
i.e., its external derivative equals zero. 
The following formula holds,
\begin{align*}
\dd \omega_u[\xi_0,\xi_1,\xi_2] &= \D_u\omega_u(\xi_1,\xi_2)[\xi_0] - \D_u\omega_u(\xi_0,\xi_2)[\xi_1]
+ \D_u\omega_u(\xi_0,\xi_1)[\xi_2]\\
& \quad - \omega_u([\xi_0,\xi_1]_u,\xi_2) + \omega_u([\xi_0,\xi_2]_u,\xi_1) - 
\omega_u([\xi_1,\xi_2]_u,\xi_0)\,,
\end{align*}
where $\D_u\omega_u(\xi_1,\xi_2)[\xi_0]$ denotes the differential of $\omega$ with respect to $u$
at point $(u;\xi_1,\xi_2)$ in the direction of $\xi_0$, and $[\cdot\,,\cdot]_u$ denotes the Poisson
bracket of vector fields at point $u$, defined by
\begin{align}
[\xi_0,\xi_1]_u &:= \phi_1\xi_0 - \phi_0\xi_1 \label{def:vectPoisson} 
=  \phi_1\pi_u(\J_u\phi_0) - \phi_0\pi_u(\J_u\phi_1)\,.
\end{align}
Directly from the definition \eqref{def:diform} we calculate   
\begin{align*}
\D_u\omega_u(\xi_1,\xi_2)[\xi_0] &:=  \frac{\dd}{\dd s}\omega_{u+s\xi_0}(\xi_1,\xi_2)\Big|_{s=0}\\
&= -\frac12\int_\T\frac{1}{\sqrt u}\left(\phi_2\pa_x(\sqrt{u}\phi_1) - 
 \phi_1\pa_x(\sqrt{u}\phi_2)\right)\xi_0\,\dd x\,,
\end{align*}
and analogously other two expressions. Also, by the definition
\begin{equation*}
\omega_u([\xi_0,\xi_1]_u,\xi_2) = -\int_\T \sqrt{u}\left(\phi_1\pa_x(\sqrt{u}\phi_0) 
 - \phi_0\pa_x(\sqrt{u}\phi_1)\right)\phi_2\,\dd x\,.
\end{equation*}
Then straightforward calculations yield $\dd \omega_u[\xi_0,\xi_1,\xi_2] = 0$.

Let $\entropy:\M\to\R$ be a Hamiltonian (for instance the Fisher information $\F$),
the corresponding Hamiltonian vector field $X_\Ham$ is defined
through the identity 
\begin{equation*}
\dd \Ham_u(\xi) = \omega_u(\xi,X_\Ham(u)) = -\int_\T \sqrt{u}\phi_\Ham\pa_x(\sqrt u\phi)\dd x
\end{equation*}
for all $\xi = \pi_u(\J_u\phi)\,,\ \phi\in C^\infty(\T)$ and $X_\Ham = \J_u\phi_\Ham$.
On the other hand 
\begin{align*}
\dd \Ham_u(\xi) = \int_\T\Ham'(u)\xi\,\dd x
= - \int_\T \Ham'(u)\sqrt{u}\pa_x(\sqrt{u}\phi)\dd x 
= \int_\R \sqrt{u}\pa_x\left(\sqrt{u}\Ham'(u)\right)\phi\,\dd x\,,
\end{align*}
which shows that
\begin{equation*}
X_\Ham(u) = \J_u \Ham'(u)\,.
\end{equation*}


\section{Well-posedness and long time behavior of nonnegative weak solutions}
\label{sec:weak}
\subsection{Existence of weak solutions --- proof of Theorem \ref{tm:exist}}
Construction of the weak solution is divided into three main steps: analysis of the time discrete
problem, passage to the limit $\tau\downarrow0$ with the time step $\tau$ and discussion on uniqueness. 
\subsubsection{Time discrete equation}\label{sec:semidiscrete} Let $\tau>0$ be given time step. We discretize equation 
(\ref{1.eq:eDLSS}) in time by means of the implicit Euler scheme.
The semi-discrete equation then reads 
\begin{equation}\label{3.eq:semi_discrete}
\frac{1}{\tau}(u - u_0) 
= - \frac12\left(u\left(\log u\right)_{xx}\right)_{xx} + 
	2\delta\left(u^{3/4}(u^{1/4})_{xx}\right)_x\quad\text{on }\T\,,
\end{equation}
where $u_0\geq 0$ a.e.~is given. Our aim is to solve nonlinear equation (\ref{3.eq:semi_discrete}) by means
of the fixed point method. For this purpose we divide our procedure into several steps.
First we linearize and regularize equation (\ref{3.eq:semi_discrete}).
Linearization is performed by the change of variables $y = \log u$, while we regularize it by adding
an elliptic operator $- \eps(\pa_x^6y - y) - \eps((\log u)_{x}^4y_x)_x$,
where $\eps>0$ is a small parameter. For given strictly positive function $u = e^z$ and 
$\sigma\in[0,1]$
we then relax the above equation (\ref{3.eq:semi_discrete}) 
into a linear elliptic equation in terms of $y$:
\begin{equation}\label{3.eq:fixpoint}
\frac{\sigma}{\tau}(e^z - u_0) 
= - \frac12\left(e^z\, y_{xx}\right)_{xx} + 2\delta\sigma\left(e^z\left(\frac{z_{xx}}{4} + \frac{z_x^4}{16} \right) \right)_x
+ \eps\left(\pa_x^6y + \left(z_{x}^4\,y_{x}\right)_x  - y \right) \quad\text{on }\T\,.
\end{equation}
More precisely, for fixed $\eps > 0$ we have formulated the fixed point mapping
$S_\eps : H^2(\T)\times[0,1]\to H^2(\T)$ defined by $S(z,\sigma) = y$, where $y\in H^3(\T)$
is the unique solution to the elliptic problem (\ref{3.eq:fixpoint}). 

The existence and uniqueness
of $y$ follows directly from the Lax-Milgram lemma. By standard arguments we also assert continuity
and compactness of the operator $S_\eps$. Observe that $S_\eps(z,0)=0$ for every $z\in H^2(\T)$,
while fixed points of $S_\eps(\cdot,1)$ will be solutions to the regularization of equation 
(\ref{3.eq:semi_discrete}).
In order to apply the Leray-Schauder fixed point theorem and conclude the existence of solutions
we need a uniform (in $\sigma$) estimate on the set of fixed points of $S_\eps(\cdot,\sigma)$ for
all $\sigma\in[0,1)$.
Let $y\in H^3(\T)$ be such a fixed point.
Employing the test function $\phi = 2 - 2e^{-y/2}\in H^3(\T)$ in the weak formulation of (\ref{3.eq:fixpoint})
and using the estimate (\ref{2.ineq:epi}) we find
\begin{align}\label{3.ineq:sigma}
\frac{\sigma}{\tau}\E(e^y) + 4\int_\T \left(e^{y/4}\right)_{xx}^2\dd x
+ \eps\kappa\int_\T e^{-y/2}\left(y_{xxx}^2  + y_x^6 \right)\dd x
\leq \frac{\sigma}{\tau}\E(u_0)\,,
\end{align}
where $\kappa>0$ is some positive constant. Here we used the pointwise inequality
$a(1-e^{-a/2}) \geq 0$ for all $a\in\R$. At this point it also becomes apparent why 
the regularization part contains the term $\eps\left(z_{x}^4\,y_{x}\right)_x$. Namely,
the linear regularization solely, would destroy the dissipation structure of the original equation,
while adding this nonlinear term ensures the above uniform estimate.
Estimate (\ref{3.ineq:sigma}) also implies $\|e^{y/2}-1\|_{L^2(\T)}\leq C$, which provides 
$\|e^{y/4}\|_{L^2(\T)}\leq C$, where $C$ is independent of $\sigma$. 
The latter conclusion together with (\ref{3.ineq:sigma}) implies in further 
$\|e^{y/4}\|_{H^2(\T)}\leq C$, while the continuity of the Sobolev embedding 
$H^2(\T)\hookrightarrow L^\infty(\T)$ asserts $\|e^{y/4}\|_{L^\infty(\T)}\leq C$.
Combining this again with (\ref{3.ineq:sigma}) gives $\sqrt{\eps}\|y_{xxx}\|_{L^2(\T)} \leq C$,
while using the test function $\phi = 1$ in (\ref{3.eq:fixpoint}) gives 
 $\eps|\int_\T y\,\dd x |\leq C$. The last two statements with help of the Poincar\'e inequality
 eventually provides the uniform ($\sigma$-independent) estimate $\|y\|_{H^3(\T)}\leq C$, 
 i.e.~$\|y\|_{H^2(\T)}\leq C$.
 The last constant $C$ depends on $\eps$, but this is not an issue here.
 With this uniform estimate we conclude the existence of a fixed point $y_\eps$ of the mapping
 $S_\eps(\cdot,1)$, and thus, the existence of a weak solution of 
 \begin{equation}\label{3.eq:regularized}
\frac{1}{\tau}(u_\eps - u_0) 
= - \frac12\left(u_\eps\, y_{\eps,xx}\right)_{xx} + 2\delta\left(u_\eps^{3/4}(u^{1/4}_\eps)_{xx}\right)_x
+ \eps\left(\pa_x^6y_\eps + \left((y_{\eps,x})^5\right)_x  - y_\eps \right) \quad\text{on }\T\,,
\end{equation}
where $u_\eps = e^{y_\eps}$, and therefore $u_\eps$ is strictly positive.

Our next step is deregularization of equation (\ref{3.eq:regularized}), i.e.~we consider the
limit of all terms in (\ref{3.eq:regularized}) as $\eps\downarrow0$. Again from estimate
(\ref{3.ineq:sigma}) (with $\sigma=1$) and above discussion we conclude 
\begin{equation}\label{3.ineq:eps}
\sqrt{\eps}\|y_{\eps,xxx}\|_{L^2(\T)} + \eps^{1/6}\|y_{\eps,x}\|_{L^6(\T)} + \sqrt{\eps}\|y_{\eps}\|_{L^2(\T)}\leq C\,,
\end{equation}
where $C>0$ is independent of $\eps$.
The latter $L^2$ estimate essentially follows from (\ref{3.eq:regularized}) utilizing the pointwise
inequality $2a(1-e^{-a/2}) \geq a^2 - e^{a/2} - 5$ for all $a\in\R$.
As we already discussed above, we have
\begin{equation}
\|u_\eps^{1/4}\|_{H^2(\T)} \leq C\,,
\end{equation}
which implies (up to a subsequence)
\begin{equation}
u_\eps^{1/4} \rightharpoonup u^{1/4}\quad\text{weakly in }H^2(\T)\,.
\end{equation}
Since $u_\eps$ is strictly positive and smooth enough, we can write
\begin{equation}
u_\eps\, y_{\eps,xx} =
4u^{1/2}_\eps\left(u^{1/4}_\eps(u^{1/4}_\eps)_{xx} - (u^{1/4}_\eps)_x^2\right)\,,
\end{equation}
and invoking compactness of Sobolev embeddings 
$H^2(\T)\hookrightarrow L^\infty(\T)$ and $H^2(\T)\hookrightarrow W^{1,4}(\T)$, we conclude:
\begin{align*}
u_\eps\, y_{\eps,xx} &\rightharpoonup 
4u^{1/2}\left(u^{1/4}(u^{1/4})_{xx} - (u^{1/4})_x^2\right)
\quad\text{weakly in }\, L^2(\T)\,,\\
\left(u_\eps^{3/4}(u^{1/4}_\eps)_{xx}\right)_x &\rightharpoonup
\left(u^{3/4}(u^{1/4})_{xx}\right)_x\quad\text{weakly in }\, L^2(\T)\,,\\
u_\eps &\to u \quad\text{strongly in }\, L^\infty(\T)\,.
\end{align*}
Moreover, uniform estimate (\ref{3.ineq:eps}) implies
\begin{equation*}
\eps\left(\pa_x^6y_\eps + \left((y_{\eps,x})^5\right)_x  - y_\eps \right) \to 0
\quad\text{strongly in }H^3(\T)\,,
\end{equation*}
and we finally conclude that
$u$ is a weak solution to
\begin{equation*}
\frac{1}{\tau} (u - u_0) = 
- 2\left(u^{1/2}\left(u^{1/4}(u^{1/4})_{xx} - (u^{1/4})_x^2\right)\right)_{xx} + 
2\delta\left(u^{3/4}(u^{1/4})_{xx}\right)_x\,\quad\text{on }\T\,.
\end{equation*}
Employing the test function $\phi=1$, it readily follows that $\int_\T u\,\dd x = \int_\T u_0\,\dd x = 1$. 
Standard arguments of weak lower semicontinuity of both entropy and the entropy production bound
provide the discrete entropy production inequality
\begin{equation}\label{ineq:depi}
\E(u) + 4\tau \int_\T (u^{1/4})_{xx}^2\dd x \leq \E(u_0)\,,
\end{equation}
which is essential for the next step of the procedure.

\subsubsection{Passage to the limit $\tau\downarrow0$}
Let the time horizon $T>0$ and the time step $\tau > 0$ be such that $T/\tau = N\in\N$. 
Then using recursively procedure from the previous step we construct solutions $u_\tau^k$
satisfying
\begin{equation}\label{3.eq:weak_k}
\frac{1}{\tau} (u_\tau^k - u_\tau^{k-1}) = 
- 2\left((u_\tau^k)^{1/2}\left((u_\tau^k)^{1/4}((u_\tau^k)^{1/4})_{xx} 
- ((u_\tau^k)^{1/4})_x^2\right)\right)_{xx} + 
2\delta\left((u_\tau^k)^{3/4}((u_\tau^k)^{1/4})_{xx}\right)_x\,
\end{equation}
and 
\begin{equation}\label{3.ineq:epi_k}
\E(u_\tau^k) + 4\tau \int_\T ((u_\tau^k)^{1/4})_{xx}^2\,\dd x \leq \E(u_\tau^{k-1})\,
\end{equation}
for $k=1,\ldots,N$.
Defining the step function $u_\tau$ according to 
\begin{align*}
u_\tau(0) &:= u_0\,,\\ 
u_\tau(t) &:= u_\tau^k\,,\quad(k-1)\tau < t \leq k\tau\,, \ k = 1,\ldots, N\,,
\end{align*}
the sequence of equations in (\ref{3.eq:weak_k}) sums up to 
\begin{align}\label{3.eq:weak_step}
\frac{1}{\tau}\int_0^T\!\!\!\int_\T(u_\tau  - \sigma_\tau u_\tau)\phi\,\dd x\dd t
&= -2 \int_0^T\!\!\!\int_\T\left(u_\tau^{1/2}\left(u_\tau^{1/4}(u_\tau^{1/4})_{xx} 
- (u_\tau^{1/4})_x^2\right)\right)\phi_{xx}\,\dd x\dd t\\	\nonumber
&\qquad - 2\delta\int_0^T\!\!\!\int_\T u_\tau^{3/4}(u_\tau^{1/4})_{xx}\phi_x\,\dd x\dd t\,,
\end{align}
for all test functions $\phi\in L^1(0,T;H^2(\T))$, while inequality (\ref{3.ineq:epi_k})
results in 
\begin{equation}\label{3.ineq:epi_tau}
\E(u_\tau^N) + 4\int_0^T\!\!\!\int_\T (u_\tau^{1/4})_{xx}^2\,\dd x\dd t \leq \E(u_0)\,.
\end{equation}
The last inequality directly implies uniform (in $\tau$) estimates
\begin{align*}
\|(u_\tau^{1/4})_{xx}\|_{L^2(0,T;L^2(\T))}&\leq C\,,\\
\|\sqrt{u_\tau} - 1\|_{L^\infty(0,T;L^2(\T))}&\leq C\,,
\end{align*}
which jointly imply $\|u_\tau^{1/4}\|_{L^2(0,T;H^2(\T))} \leq C$ and therefore we have
(up to a subsequence) the weak convergence of the sequence $(u_\tau^{1/4})$ 
to some $v\in L^2(0,T;H^2(\T))$, i.e.
\begin{equation}\label{3.eq:v}
u_\tau^{1/4} \rightharpoonup v\quad \text{weakly in }L^2(0,T;H^2(\T))\,.
\end{equation}
Estimate $\|\sqrt{u_\tau}\|_{L^\infty(0,T;L^2(\T))}\leq C$ implies
$\|u_\tau^{1/4}\|_{L^\infty(0,T;L^4(\T))}\leq C$. Furthermore, the Gagliardo-Nirenberg
inequality provides
\begin{equation*}
\|u_\tau^{1/4}(t)\|_{L^\infty(\T)} 
\leq C\|u_\tau^{1/4}(t)\|_{H^2(\T)}^{1/7}\, \|u_\tau^{1/4}(t)\|_{L^4(\T)}^{6/7}
\end{equation*}
for a.e.~$t\in(0,T)$. Therefore, the Young inequality gives the uniform bound
\begin{equation}
\|u_\tau^{1/4}\|_{L^7(0,T;L^\infty(\T))} \leq C\,.
\end{equation}
The above obtained estimates are now sufficient to conclude
the uniform a priori estimate on the sequence of finite differences
\begin{equation}\label{3.ineq:dt_u_tau}
\tau^{-1}\|u_\tau - \sigma_\tau u_\tau\|_{L^{1}(\tau,T;H^{-2}(\T))} 
\leq C\,,
\end{equation}
where $\sigma_\tau u_\tau = u_\tau(\cdot - \tau)$ denotes the left shift operator in time.

Next we want to prove
\begin{equation}\label{3.ineq:u_tau}
\|u_\tau\|_{L^2(0,T;W^{2,1}(\T))} \leq C\,.
\end{equation}
Notice that 
\begin{equation*}
(u_\tau^{1/2})_{x} = 2u_\tau^{1/4}(u_\tau^{1/4})_x\,,\quad (u_\tau^{1/2})_{xx} = 2\left(u_\tau^{1/4}(u_\tau^{1/4})_{xx} + (u_\tau^{1/4})_x^2 \right)\,,
\end{equation*} 
which, using the previous estimates, implies $\|u_\tau^{1/2}(t)\|_{H^2(\T)}\leq C$ for
a.e.~$t\in(0,T)$. Employing the Lions-Villani theorem provides
$\|u_\tau^{1/4}(t)\|_{W^{1,4}(\T)}^2\leq C_{LV}\|u_\tau^{1/2}(t)\|_{H^2(\T)}\leq C$ 
(cf.~\cite[Lemma 26]{BJM13}).
Integrating the last inequality over $(0,T)$ we get the uniform estimate
\begin{equation*}
\|u_\tau^{1/4}\|_{L^4(0,T;W^{1,4}(\T))}\leq C\,.
\end{equation*}
Now observe that
\begin{equation*}
(u_\tau)_x = 4 u_\tau^{3/4}(u_\tau^{1/4})_x\,,\quad 
(u_\tau)_{xx} = 4 u_\tau^{3/4}(u_\tau^{1/4})_{xx} + 12u_\tau^{1/2}(u_\tau^{1/4})_x^2\,,
\end{equation*} 
which, again using the above estimates, implies the desired bound (\ref{3.ineq:u_tau}).
Having at hand estimates (\ref{3.ineq:dt_u_tau}) and (\ref{3.ineq:u_tau}) we can invoke
the Aubin-Lions lemma \cite{CJL14} and conclude the strong convergence
(on a subsequence as $\tau\downarrow0$)
\begin{equation}\label{3.eq:strong_u_t}
u_\tau \to u \quad\text{strongly in }L^2(0,T;W^{1,6}(\T))\,.
\end{equation}

In order to pass to the limit as $\tau\downarrow0$ in (\ref{3.eq:weak_step}), we need to explore 
some more convergence results. First, using (\ref{3.eq:strong_u_t}) we identify in
(\ref{3.eq:v}) $v = u^{1/4}$. Then, employing \cite[Proposition 6.1]{JuMi09} on (\ref{3.eq:v}) and
(\ref{3.eq:strong_u_t}) we conclude
\begin{equation}\label{3.eq:34}
u_\tau^{3/4}\to u^{3/4}\quad\text{strongly in }L^8(0,T;W^{1,8}(\T))\,,
\end{equation} 
and similarly
\begin{equation*}
u_\tau^{1/2}\to u^{1/2}\quad\text{strongly in }L^{12}(0,T;W^{1,12}(\T))\,.
\end{equation*} 
Using a stronger version of the entropy production inequality, namely
\begin{equation*}
\E(u_\tau^N) + \kappa\int_0^T\!\!\!\int_\T\left( (u_\tau^{1/4})_{xx}^2
 + (u_\tau^{1/8})_x^4 \right)\dd x\dd t \leq \E(u_0)\,
\end{equation*}
for some $\kappa>0$, which can be proved in the same fashion as the basic one, 
we immediately have the uniform bound
\begin{equation*}
\|u_\tau^{1/8}\|_{L^4(0,T;W^{1,4}(\T))} \leq C\,.
\end{equation*}
Combining the latter with (\ref{3.eq:34}) we conclude 
(again using \cite[Proposition 6.1]{JuMi09})
\begin{equation*}
u_\tau^{1/4}\to u^{1/4}\quad\text{strongly in }L^2(0,T;H^2(\T))\,.
\end{equation*}
The above convergence results are now sufficient to pass to the limit in (\ref{3.eq:weak_step}).

\subsubsection{Uniqueness} 
Estimates of the previous subsection provide $u^{1/4}\in L^2(0,T;H^2(\T))$ and 
$u^{1/2}\in L^2(0,T;H^2(\T))$, which is precisely the required regularity in \cite{Fis13},
which ensures the uniqueness of global weak solutions for equation (\ref{1.eq:eDLSS}) with $\delta=0$.
The established regularity of $u$ is sufficient to make the calculations of \cite{Fis13} 
rigorous also for the third-order term in (\ref{1.eq:eDLSS}). 
Thus, the weak solution constructed above also satisfies equation (cf.~\cite[Lemma 15]{Fis13})
\begin{align}
-\int_0^T\!\!\!\int_\T \sqrt u\pa_t\phi\,\dd x\dd t & - \int_\T \sqrt{u_0}\phi(\cdot,0)\,\dd x
= -\frac12 \int_0^T\!\!\!\int_\T (\sqrt u)_{xx}\phi_{xx}\,\dd x\dd t\\	\nonumber
&+ \frac12\int_0^T\!\!\!\int_\T \frac{\phi}{\sqrt u}(\sqrt u)_{xx}^2\,\dd x\dd t
- \frac{\delta}{2}\int_0^T\!\!\!\int_\T (\sqrt u)_{xx}\phi_{x}\,\dd x\dd t\,
\end{align}
for all test functions $\phi\in L^\infty(0,T;W^{2,\infty}(\T))\cap W^{1,1}(0,T;L^\infty(\T))$
satisfying $\phi(\cdot,T)\equiv0$.

This finishes the proof of Theorem \ref{tm:exist}.

\subsection{Large time behavior of weak solutions}
The discrete entropy production inequality
\begin{equation*}
\E(u_\tau^N) + 4\int_0^T\!\!\!\int_\T (u_\tau^{1/4})_{xx}^2\,\dd x\dd t \leq \E(u_0)\,
\end{equation*}
provides the Lyapunov stability. Namely, using the weak lower semicontinuity 
of the fun\-cti\-o\-nal on the left hand side, for the weak solution $u$ of (\ref{1.eq:eDLSS}) we have
\begin{equation*}
\sup_{t\in (0,T)}\E(u(t)) \leq \E(u_0)\,.
\end{equation*}
In order to conclude a stronger result, one needs an entropy -- entropy production inequality
which stems from a Beckner type inequality. Unfortunately, the ``global'' Beckner
inequality of type
\begin{equation}\label{3.ineq:Beckner}
\frac{p}{p-1}\left(\int_\T f^2\dd x - \left( \int_\T f^{2/p}\dd x\right)^p  \right)
\leq C_{B}\int_\T (f_{xx})^2\dd x
\end{equation} 
is valid for $p\in (1,2]$. In order to apply such inequality in our case, we would need
inequality (\ref{3.ineq:Beckner}) with $p=1/2$, which is out of the scope here. 
Therefore, we rely on an ``asymptotic'' Beckner inequality proved in \cite[Corollary 2]{CDGJ06}:
for any $p > 0$, $q \in \R$ and $\eps_0 > 0$, there exists a positive constant $C$ 
(depending on $p,q$ and $\eps_0$)
such that, for any $\eps\in(0,\eps_0]$
\begin{equation}\label{ineq:aBeck}
\Sigma_{p,q}(f):=\frac{1}{pq(pq-1)}\left(\int_{\T} f^q\dd x - \left(\int_{\T} f^{1/p}\dd x \right)^{pq} \right)
\leq \frac{1+C\sqrt\eps}{32p^2\pi^4}\int_\T\left(f_{xx}\right)^2\dd x
\end{equation}
for all $f\in \mathcal{X}_\eps^{p,q} = \left\{ f\in H^2(\T)\ : 
\ f\geq0\text{ a.e.},\ \Sigma_{p,q}(f)\leq \eps \text{ and }\int_{\T} f^{1/p}\dd x = 1  \right\}
$. 
\begin{proof}[Proof of Theorem \ref{tm:ltb}] Let $u_0$ be given initial datum, $\tau>0$ and 
let $u_\tau^1, u_\tau^2, \ldots$ be 
the sequence of solutions to the semi-discrete problem constructed in section (\ref{sec:semidiscrete}).
The discrete entropy production inequality (\ref{ineq:depi}) provides
\begin{equation}\label{ineq:depi}
\E(u_\tau^k) + 4\tau \int_\T ((u^k_\tau)^{1/4})_{xx}^2\dd x \leq \E(u_\tau^{k-1})\,,\quad\forall k\in\N\,.
\end{equation}
Employing inequality (\ref{ineq:aBeck}) with $p=1/4$, $q=2$, $\eps_0 = \E(u_0)$ and $f = (u^k_\tau)^{1/4}$
it readily follows 
\begin{equation}\label{ineq:deepi}
\E(u_\tau^k) \leq \frac{1+C\sqrt{\eps_0}}{2\pi^4}\int_\T ((u^k_\tau)^{1/4})_{xx}^2\dd x
\,,\quad\forall k\in\N\,.
\end{equation}
Combining (\ref{ineq:depi}) and (\ref{ineq:deepi}) yields
\begin{equation*}
\E(u_\tau^k) + \frac{8\pi^4\tau}{1+C\sqrt{\eps_0}}\E(u_\tau^k) 
\leq \E(u_\tau^{k-1})\,,\quad\forall k\in\N\,,
\end{equation*}
which passing to the limit $\tau\downarrow0$ implies
\begin{equation*}
\E(u(t))\leq \E(u_0)e^{-\frac{8\pi^4t}{1+C\sqrt{\eps_0}}}\,,\quad t>0\,.
\end{equation*}
The well known relation between the $L^1$ and the Hellinger distance finally provides
\begin{equation*}
\|u(t)-1\|_{L^{1}(\T)} \leq 2\H(u(t),1) = \sqrt{\E(u(t))} \leq \sqrt{\E(u_0)}e^{-\frac{4\pi^4t}{1+C\sqrt{\eps_0}}}
\,,\quad t>0\,.
\end{equation*}
\end{proof}

\section{A structure preserving numerical scheme}

\subsection{Introduction of the scheme}\label{sec:41}
In this section we devise a numerical scheme for equation (\ref{1.eq:eDLSS}), which 
respects its basic properties: nonnegativity, mass conservation and the dissipation of the
 Fisher information on the discrete level. More precisely, the scheme is a discretization of 
(\ref{2.eq:gf_form}) with the time discretization inspired by (\ref{3.eq:semi_discrete_new}).
It is a discrete variational derivative (DVD) type scheme, which is a slight modification of the
scheme from \cite{BEJ14} proposed for the original DLSS equation. 
The advantage of the method proposed here is the error estimate given in terms of the
discrete Hellinger distance.

Let $\T_N = \{x_i\ : \ i=0,\ldots,N,\ x_0\cong x_N\}$ denotes an
equidistant discrete grid of mesh size $h$ on the one dimensional torus $\T
\cong [0,1)$ and let the vector $U^k\in \R^N$ with components $U_i^k$ approximates
the solution $u(t_k,x_i)$ for $i=0,\ldots,N-1$ and
$k\geq0$. 
We will use the following standard finite difference operators. 
For $U\in\R^N$ define:

\begin{tabular}{ll}
forward difference: & $\delta_i^+U = h^{-1}(U_{i+1} - U_{i})$,\\
backward difference: & $\delta_i^-U = h^{-1}(U_i - U_{i-1})$,\\
central difference: & $\delta_i^{\langle1\rangle}U = (2h)^{-1}(U_{i+1} - U_{i-1})$,\\
2nd order central difference: & 
$\delta_i^{\langle2\rangle}U = \delta_i^+\delta_i^-U = h^{-2}(U_{i+1} - 2U_i + U_{i-1})$.
\end{tabular} 

\noindent To approximate the
integral of one-periodic functions $w$, we use the first-order quadrature rule
$\sum_{i=0}^{N-1}w(x_i)h$. This rule is in fact of the second order,
since due to the periodic boundary conditions it coincides with the trapezoidal rule
$(w(x_0) + w(x_N))h/2 + \sum_{i=1}^{N-1}w(x_i)h$.

The first step is to define a discrete analogue of the Fisher information
$\F_{\dd}:\R^N\to\R$ as an approximation of the true Fisher information $\F$. 
The basic idea of DVD methods is to perform a discrete variation procedure 
and calculate the corresponding discrete variational derivative.
We approximate the Fisher information $\F(u)$ by
\begin{equation}\label{def:discreteFisher}
 \Fd[U] = \frac12\sum_{i=0}^{N-1}\big((\delta_i^+V_i)^2 +
(\delta_i^-V_i)^2\big)h\,,
\end{equation}
where $U\in\R^N$ and $V_i = \sqrt{U_i}$ for $i=0,\ldots,N-1$. 
Applying the discrete variation procedure and using
summation by parts formula (see \cite[Proposition 3.2]{FuMa10}) for periodic
boundary conditions, we calculate:
\begin{align*}
\Fd[U^{k+1}] &- \Fd[U^k] = \frac12\sum_{i=0}^{N-1}\left((\delta_i^+V_i^{k+1})^2
- (\delta_i^+V_i^k)^2 + (\delta_i^-V_i^{k+1})^2 - (\delta_i^-V_i^k)^2\right)h\\
&= \frac12\sum_{i=0}^{N-1}\big(\delta_i^+(V_i^{k+1} + V_i^k)\delta_i^+(V_i^{k+1}
- V_i^k) +  \delta_i^-(V_i^{k+1} + V_i^k)\delta_i^-(V_i^{k+1} -
V_i^k)\big)h\\
&= -\sum_{i=0}^{N-1}\delta_i^{\langle 2\rangle}(V_i^{k+1} + V_i^k)(V_i^{k+1} -
V_i^k)h
= -\sum_{i=0}^{N-1}\frac{\delta_i^{\langle 2\rangle}(V_i^{k+1} +
V_i^k)}{V_i^{k+1} + V_i^k}(U_i^{k+1} - U_i^k)h
\end{align*}
for $k\geq0$.

The discrete variational derivative, denoted by $\delta
\Fd(U^{k+1},U^k)\in\R^N$, is then defined componentwise by
\begin{equation}\label{discr.var}
\delta \Fd(U^{k+1},U^k)_i := -\frac{\delta_i^{\langle
2\rangle}(V_i^{k+1} + V_i^k)}{V_i^{k+1} + V_i^k}\,,\quad i=0,\ldots,N-1\,,
\end{equation}
and the main point is that the discrete chain rule holds
\begin{align*}
\Fd[U^{k+1}] - \Fd[U^k] &= \sum_{i=0}^{N-1}\delta
\Fd(U^{k+1},U^k)_i(U_i^{k+1} - U_i^k)h\,.
\end{align*}

Having this at hand, the DVD scheme for the corrected DLSS equation is defined by the following 
nonlinear system with unknowns
$V^{k+1}_i = \sqrt{U^{k+1}_i}$:
\begin{align}\label{sh.dvdm}
\frac{1}{\tau}(V_i^{k+1} - V_i^k) = \frac{1}{2W^{k+1/2}_i}\delta_i^+\left(W^{k+1/2}_iW^{k+1/2}_{i-1}\delta_i^-
\left(\delta \Fd(W^{k+1/2})_i\right)\right)\\  - 
\frac{\delta}{2} \delta_i^{\langle1\rangle}\left(W^{k+1/2}_i\delta \Fd(W^{k+1/2})_i\right)\,,
\nonumber
\end{align}
for all $i=0,\ldots,N-1\,,\ k\geq 0\,,$ 
where $W^{k+1/2} = (V^{k+1} + V^k)/2$. 

\subsection{Convergence analysis --- proof of Theorem \ref{tm:dvds}}
Basic properties of the scheme
conservation of mass and dissipation of the discrete Fisher information follow directly
from the construction of the scheme, summation by parts and the above discrete chain rule.
Our main aim is to prove the convergence of the scheme. For this purpose we first prove 
the monotonicity of the following discrete operator $\Ad :\R^N_+\to\R^N$ defined by
\begin{equation}\label{4.def:monoton}
\Ad(W)_i = \frac{1}{W_i}\delta_i^+\left(W_iW_{i-1}\delta_i^-\delta\Fd(W)_i\right)\,,
\quad i=0,\ldots,N-1\,.
\end{equation}
Operator $\Ad$ is a discrete analogue of the differential operator 
\begin{equation*}
\A(w) = \frac{1}{w}\left(w^2\left(\frac{w_{xx}}{w}\right)_x\right)_{x}
\end{equation*}
whose monotonicity has been shown in \cite{JuPi01}.
\begin{proposition}
Operator $\Ad :\R^N_+\to\R^N$ defined by (\ref{4.def:monoton}) is monotone.
\end{proposition}
\begin{proof}
Let $w,W\in\R^N$ be arbitrary vectors from the cone $\R_+^N$. 
Applying the summation by parts formula twice we compute
\begin{align*}
(\Ad(w)-\Ad(W))&\cdot(w-W) \\
&= -h\sum_{i=0}^{N-1}\left(\delta \Fd(w) - \delta\Fd(W)\right)_i
\delta_i^+\left(w_iw_{i-1}\delta_i^-\left(\frac{w_i-W_i}{w_i} \right) \right) \\
&\  - \delta\Fd(W)_i\delta_i^+\left(W_iW_{i-1}\delta_i^-\left(\frac{w_i-W_i}{W_i} \right)
- w_iw_{i-1}\delta_i^-\left(\frac{w_i-W_i}{w_i} \right) \right)\,.
\end{align*}
Employing discrete differentiation rules:
\begin{align*}
w_iw_{i-1}\delta_i^-\left(\frac{w_i-W_i}{w_i} \right) &= w_{i-1}\delta_i^-(w-W) - (\delta_i^-w)(w-W)_i\,,\\
\delta_i^+(wW) &= (\delta_i^+w)W_{i+1} + w_i(\delta_i^+W) = w_{i+1}(\delta_i^+W) + W_i(\delta_i^+w) \,,\\
\delta_i^+(w_{i-1}\delta_i^-(w-W)) &= w_iW_i\left(\frac{\delta_i^{\langle2\rangle}w}{w} - 
\frac{\delta_i^{\langle2\rangle}W}{W}\right) = -w_iW_i\left( \delta \Fd(w) - \delta\Fd(W) \right)_i\,,
\end{align*}
we find 
\begin{align*}
(\Ad(w)-\Ad(W))\cdot(w-W) = h\sum_{i=0}^{N-1}w_iW_i\left(\delta \Fd(w) - \delta\Fd(W)\right)^2_i
\geq 0\,,
\end{align*}
which proves the monotonicity of $\Ad$. 
\end{proof}
With the help of operator $\Ad$ the discrete scheme
(\ref{sh.dvdm}) can be written as
\begin{equation}\label{4.sh:operator}
\frac{1}{\tau}(V_i^{k+1} - V_i^k) = -\frac12\Ad(W^{k+1/2})_i + 
2\delta \delta_i^{\langle1\rangle}\left(\delta_i^{\langle2\rangle}W^{k+1/2}\right)\,.
\end{equation}
Let $u^k\in\R^N_+$ denotes the vector of true solution values at grid points $x_i$ at time $t_k$,
i.e.~$u_i^k = u(t_k, x_i)$, and let $v^k = \sqrt{u^k}$. Then we have
\begin{equation}\label{4.sh:true_sol}
\frac{1}{\tau}(v_i^{k+1} - v_i^k) = -\frac12\Ad(w^{k+1/2})_i + 
2\delta \delta_i^{\langle1\rangle}\left(\delta_i^{\langle2\rangle}w^{k+1/2}\right) + f_i^{k+1/2}\,,
\end{equation}
where $w_i^{k+1/2} = (v_i^{k+1} + v_i^k)/2$ and values $f_i^{k+1/2}$ represent the local 
truncation error of the scheme.
Subtracting (\ref{4.sh:operator}) from (\ref{4.sh:true_sol}) we get the discrete equation for the error
vector $e^k:=v^k - V^k$ which reads
\begin{equation}\label{4.sh:error}
\frac{1}{\tau}(e_i^{k+1} - e_i^k) = -\frac12\left(\Ad(w^{k+1/2})_i - \Ad(W^{k+1/2})_i\right) + 
2\delta \delta_i^{\langle1\rangle}\left(\delta_i^{\langle2\rangle}e^{k+1/2}\right) + f_i^{k+1/2}\,.
\end{equation}
Multiplying (\ref{4.sh:error}) with $e_i^{k+1/2} = w_i^{k+1/2} - W_i^{k+1/2}$ and summing up over
$i=0,\ldots, N-1$ we find
\begin{align*}
\frac{h}{2\tau}\sum_{i=0}^{N-1}\left(\left(e_i^{k+1}\right)^2 - \left(e_i^{k}\right)^2\right)
= -\frac{h}{2}\sum_{i=0}^{N-1}\left(\Ad(w^{k+1/2})_i - \Ad(W^{k+1/2})_i\right)(w_i^{k+1/2} - W^{k+1/2}_i)\\
-2\delta h\sum_{i=0}^{N-1}\delta_i^{\langle2\rangle}e^{k+1/2}\delta_i^{\langle1\rangle}e^{k+1/2}
 + h\sum_{i=0}^{N-1}f_i^{k+1/2}e_i^{k+1/2}\,.
\end{align*} 
Employing the monotonicity of the opertor $\Ad$ and the fact that 
$\sum_{i=0}^{N-1}\delta_i^{\langle2\rangle}e^{k+1,k}\delta_i^{\langle1\rangle}e^{k+1,k}=0$
due to periodic boundary conditions and summation by parts formula, we estimate
\begin{align*}
\frac{h}{2}\sum_{i=0}^{N-1}\left(\left(e_i^{k+1}\right)^2 - \left(e_i^{k}\right)^2\right)
\leq  \tau h\sum_{i=0}^{N-1}f_i^{k+1/2}e_i^{k+1/2}\,.
\end{align*}
Using the Cauchy-Schwarz, Young and Jensen's inequalities we further estimate the right hand side
and get
\begin{align*}
\frac{h}{2}\sum_{i=0}^{N-1}\left(\left(e_i^{k+1}\right)^2 - \left(e_i^{k}\right)^2\right)
\leq  \frac{\tau h}{2}\sum_{i=0}^{N-1}\left(f_i^{k+1/2}\right)^2 
+ \frac{\tau h}{4}\sum_{i=0}^{N-1}\left( (e_i^{k+1})^2 + (e_i^k)^2\right)\,.
\end{align*}

It has been proved in \cite{BEJ14} that the local truncation error of the DVD method for
a sufficiently smooth solution is of order
$O(\tau) + O(h^2)$. Analogous arguments can be utilized here for this slightly modified scheme. 
Therefore, summing up the last inequality for $k=0,\ldots,M$ we have
\begin{align*}
(1-\tau)\frac{h}{2}\sum_{i=0}^{N-1}\left(e_i^{M+1}\right)^2 \leq 
\frac{h}{2}\sum_{i=0}^{N-1}\left(e_i^{0}\right)^2 + C(\tau^2 + h^4) + 
\frac{\tau h}{2}\sum_{k=0}^M \sum_{i=0}^{N-1}\left(e_i^{k}\right)^2\,,
\end{align*}
where $C>0$ is independent of $\tau$ and $h$.
If we assume that $e^0 = 0$, or at least small enough, then the discrete Gronwall inequality
implies (for $\tau < 1$)
\begin{equation}
\frac{h}{2}\sum_{i=0}^{N-1}\left(e_i^{M+1}\right)^2 \leq \frac{C(\tau^2 + h^4)}{1-\tau}
e^{\frac{\tau(M+1)}{2(1-\tau)}}\,\quad\text{for all } M\geq 0\,,
\end{equation}
which concludes the proof of Theorem \ref{tm:dvds}.

\subsection{Implementation and illustrative examples}
In this final subsection we illustrate numerical solutions to the corrected DLSS equation 
using the DVD method. Prior to that we expand terms of the scheme (\ref{sh.dvdm}) and obtain a 
novel form in unknowns $W = W^{k+1/2} = (V^{k+1} + V^k)/2$:
\begin{align}\label{sh.dvdW}
W_i - V^k_i &= -\frac{\tau}{4h^4}\left(W_{i+2} + 2W_i + W_{i-2} - \frac{(W_{i+1} + W_{i-1})^2}{W_i} \right)\\
 &\qquad + \frac{\tau\delta}{8h^3}\left(W_{i+2} - 2W_{i+1} + 2W_{i-1} - W_{i-2} \right)\,,
 \quad i=0,\ldots, N-1\,,\ k\geq0\,.
 \nonumber
\end{align}
Numerical solution $U^k$ of equation (\ref{2.eq:gf_form}) is then resolved according to
\begin{equation*}
U^{k+1}_i = (2W_i - V^k_i)^2\,,\quad i=0,\ldots,N-1,\,,\ k\geq0\,.
\end{equation*}

Note that system (\ref{sh.dvdW}) is easier to treat numerically than the system
(\ref{sh.dvdm}).
Numerical solutions are computed for two different initial conditions: (I)
$u_0 = M^{-1}_1(\cos(\pi x)^{16} + 0.1)$ (first column of Figure \ref{fig:log.sol}) and 
(II) $u_0 = M^{-1}_2(\cos(2\pi x)^{16} + 0.01)$
(second column of Figure \ref{fig:log.sol}), where constants $M_1, M_2>0$ are taken such 
that $u_0$ have unit mass. Different rows in Figure \ref{fig:log.sol} denote different 
dispersion parameter $\delta$, i.e.~$\delta = 1$, $\delta = 10$ and $\delta = 100$ in the first,
second and third row of Figure \ref{fig:log.sol}, respectively. In each subfigure numerical evolution 
is sketched in five time instances starting from the initial datum $u_0$.
Discretization parameters are taken to be $\tau = 10^{-6}$ and $h = 5\cdot10^{-3}$, 
and the nonlinear scheme (\ref{sh.dvdW})
is solved by the Newton's method using the solution from the previous time step as an initial guess
for the solution on the current time step. Complete algorithm is implemented in \verb1Matlab1.
Figure \ref{fig:log.sol} also illustrates convergence of numerical solutions to the constant steady  
state $u_\infty = 1$, as indicated by Theorem \ref{tm:ltb}.

\begin{figure}[!t]
\centering
\subfloat[]{
\includegraphics[width=70mm]{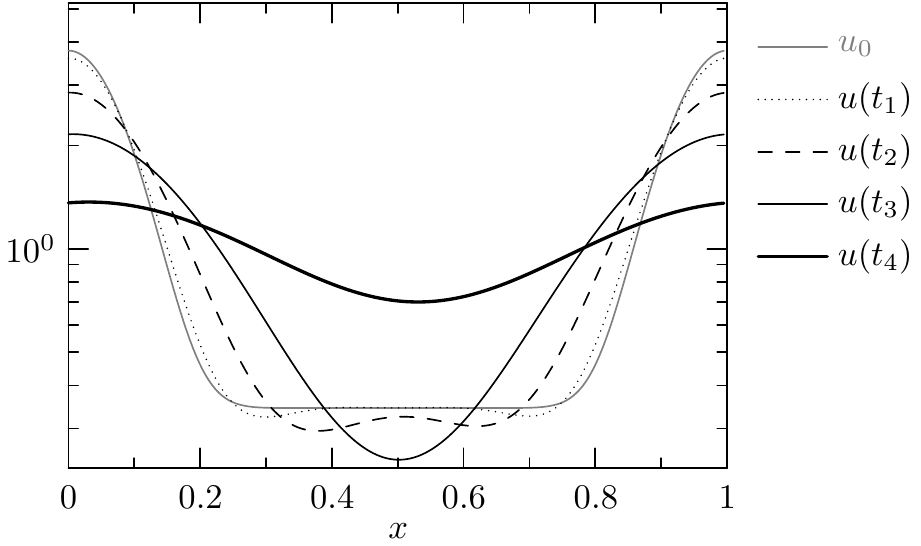}  
\label{fig:T11}
}
\hspace{5mm} 
\subfloat[]{
\includegraphics[width=70mm]{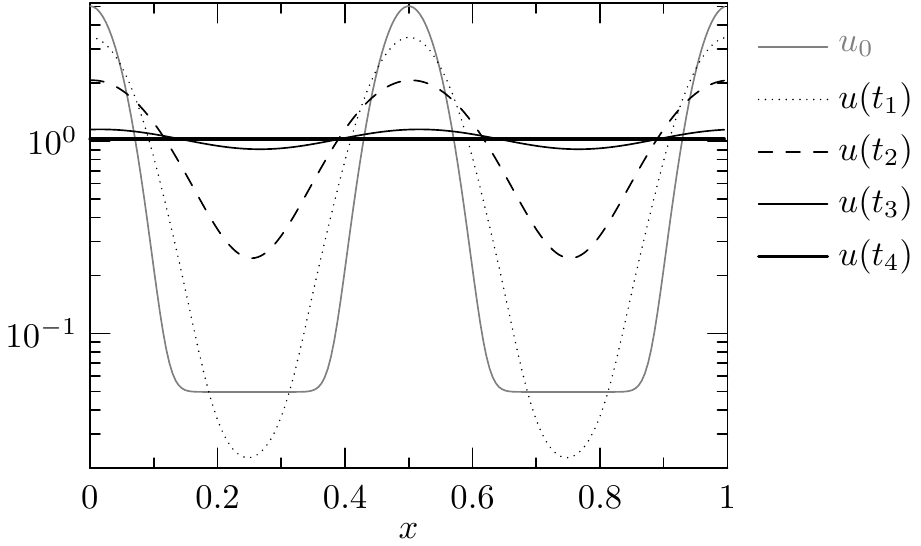}
\label{fig:T21}
}

\subfloat[]{
\includegraphics[width=70mm]{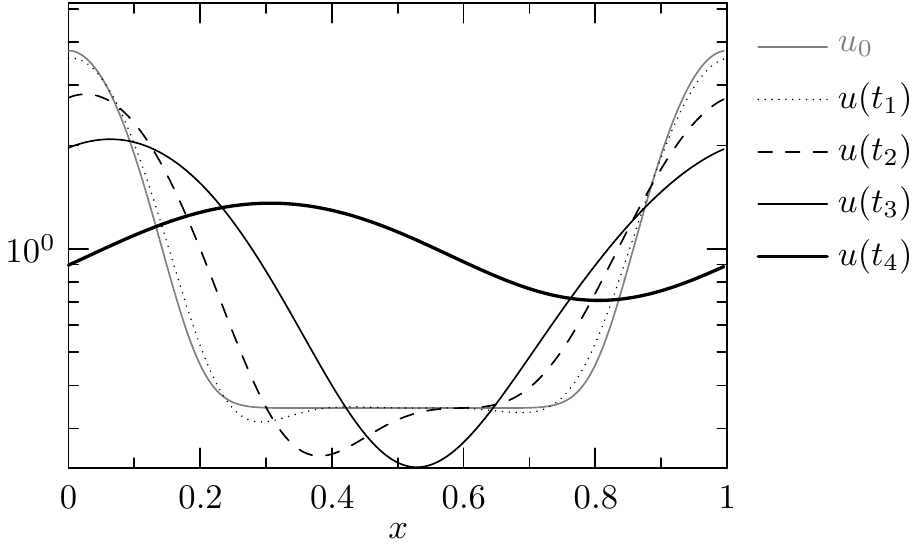}  
\label{fig:T110}
}
\hspace{5mm} 
\subfloat[]{
\includegraphics[width=70mm]{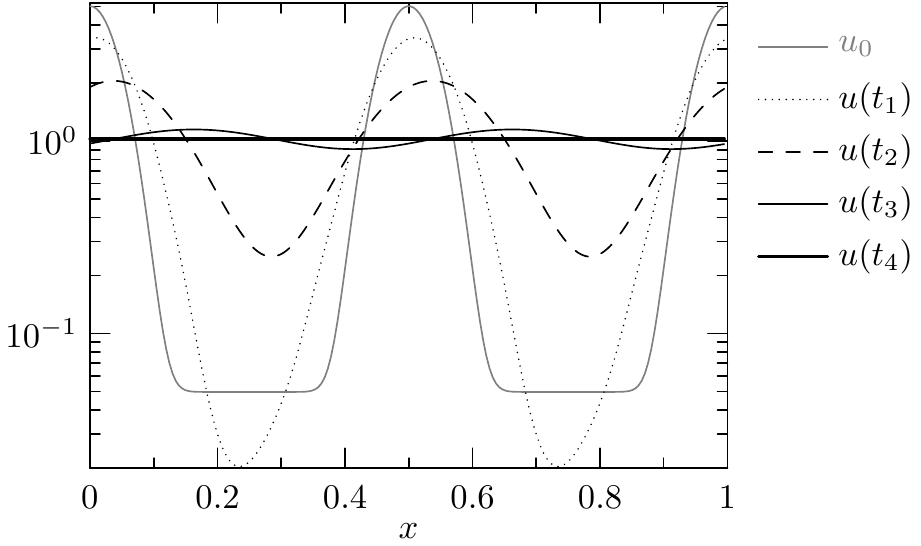}
\label{fig:T210}
}

\subfloat[]{
\includegraphics[width=70mm]{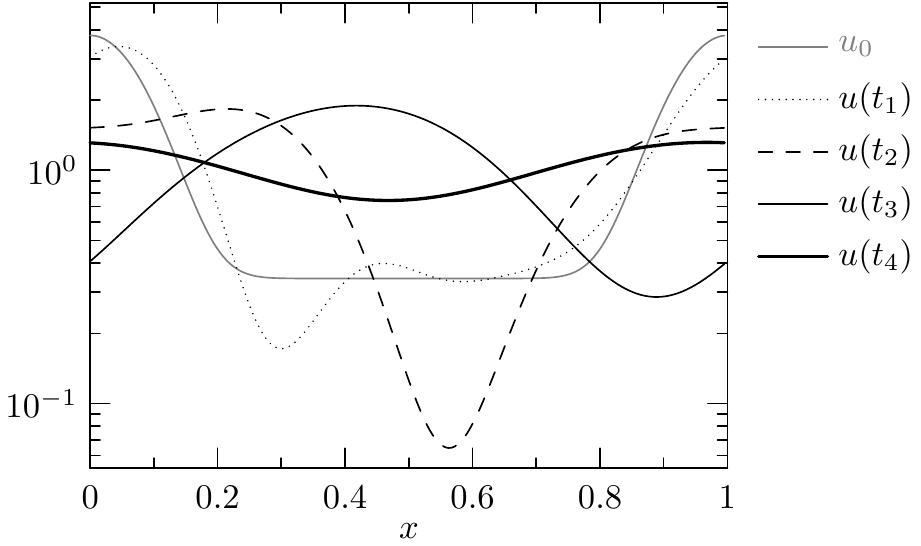}  
\label{fig:T1100}
}
\hspace{5mm} 
\subfloat[]{
\includegraphics[width=70mm]{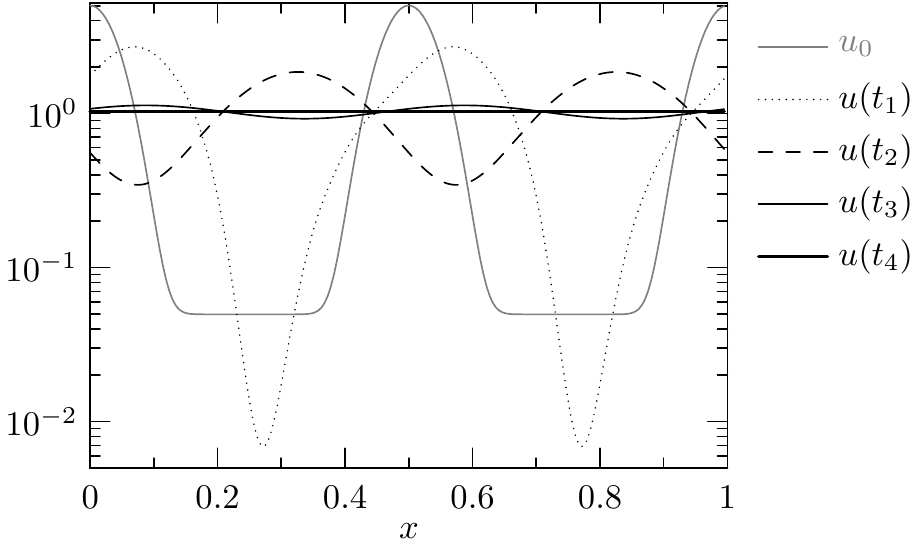}
\label{fig:T2100}
}

\caption{Numerical evolution of the corrected DLSS equation for unit mass initial datum $u_0$ at 
different time moments: $t_1 = 5\cdot10^{-6}$, $t_2 = 4\cdot10^{-5}$, $t_3 = 2\cdot10^{-4}$, and
$t_4 = 1.5\cdot10^{-3}$.}
\label{fig:log.sol}
\end{figure}

Numerical scheme (\ref{sh.dvdW}) is additionally explored by testing its numerical convergence 
rates, both in space and time. For time convergence we set $\delta = 1$, 
$u_0 = M^{-1}_1(\cos(\pi x)^{16} + 0.1)$ and $h=2\cdot10^{-3}$. The ``exact solution'' $\hat{u}$
is computed on the very fine time resolution $\tau = 10^{-9}$ and all other numerical 
solutions $U^\tau$ are compared at time instance $T = 5\cdot10^{-5}$ using the discrete
Hellinger distance $\H_{\dd}$ defined by (\ref{def.discreteHell}), 
i.e.~we calculate the error at time step $M$ corresponding to time 
instance $T$ as
\begin{equation*}
\|e^M\|_{h,l^2}:=\H_{\dd}(\hat u^M,U^M)^2 = \frac{h}{2}\sum_{i=0}^{N-1}\left(\sqrt{\hat u_i^M} 
- \sqrt{U_i^M}\right)^2\,.
\end{equation*}
Results of this numerical experiment are shown in Figure \ref{fig:conv_rates} as well as
in the Table \ref{tab:conv_rates}. One can see that they are in agreement with the
theoretical result of Theorem \ref{tm:dvds}.

\begin{figure}[!t]
\centering
\subfloat[]{
\includegraphics[height=70mm]{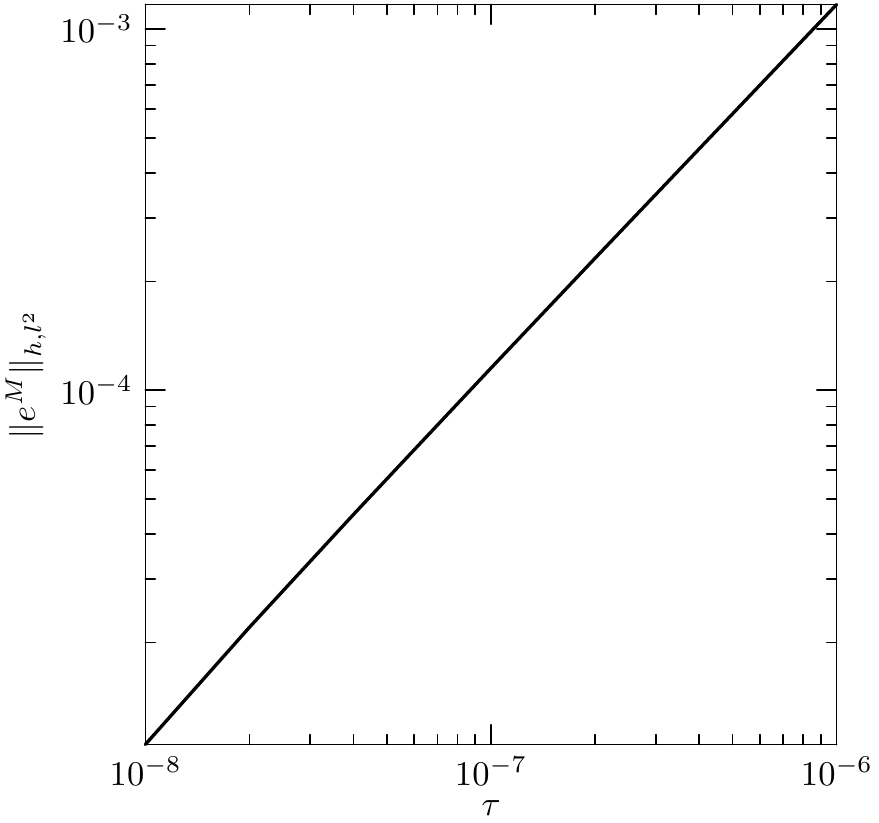}  
\label{fig:time_conv_rate}
}
\hspace{5mm} 
\subfloat[]{
\includegraphics[height=70mm]{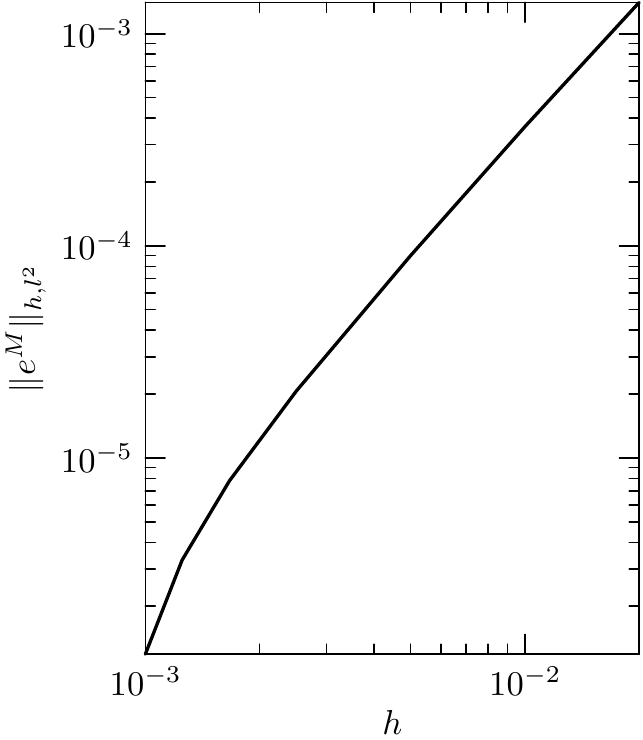}
\label{fig:space_conv_rate}
}
\caption{Errors with respect to time and space discretization parameters.}
\label{fig:conv_rates}
\end{figure}

\begin{table}[!h]
\begin{tabular}{r|l}
$\tau$ & conv.~rate\\
\hline
$10^{-8}$ & ~ \\
$2\cdot 10^{-8}$ & $1.0783$\\
$4\cdot 10^{-8}$ & $1.0381$\\
$8\cdot 10^{-8}$ & $1.0195$\\
$10^{-7}$ & $1.0185$\\
$2\cdot 10^{-7}$ & $1.0106$\\
$4\cdot 10^{-7}$ & $1.0054$\\
$5\cdot 10^{-7}$ & $1.0050$\\
$10^{-6}$ & $1.0059$
\end{tabular}
\hspace{45mm}
\begin{tabular}{r|l}
$h$ & conv.~rate\\
\hline
$10^{-3}$ & ~ \\
$1.3\cdot 10^{-3}$ & $4.5764$\\
$1.7\cdot 10^{-3}$ & $2.9940$\\
$2.5\cdot 10^{-3}$ & $2.4015$\\
$5\cdot 10^{-3}$ & $2.1202$\\
$1\cdot 10^{-2}$ & $2.0177$\\
$2\cdot 10^{-2}$ & $1.9454$
\end{tabular}
\caption{Numerical convergence rates in time (left) and space (right).}
\label{tab:conv_rates}
\end{table}

\end{document}